\documentclass[11pt,leqno]{article}

\usepackage[active]{srcltx}

\usepackage{amsfonts,latexsym,amsmath,amssymb,amsthm}
\usepackage{dsfont}
\usepackage{fullpage}
\newtheorem{theorem}{Theorem}[section]
\newtheorem{lemma}[theorem]{Lemma}

\newtheorem{proposition}[theorem]{Proposition}

\theoremstyle{definition}

\theoremstyle{remark}

\newtheorem*{note*}{Note}

\numberwithin{equation}{section}

\makeatletter

\newcommand{\rank}{\mathop{\operator@font rank}}
\newcommand{\conv}{\mathop{\operator@font conv}}
\newcommand{\vol}{\mathrm{vol}}

\newcommand{\onetagright}{\tagsleft@false}
\makeatother

\makeatletter
\newcommand{\ls}{\leqslant}
\newcommand{\gr}{\geqslant}
\makeatother

\newcommand{\vrad}{{\rm vrad}}

\usepackage{color}

\usepackage{hyperref}

\begin{document}
\small

\title{\bf Sections and projections of the outer and inner regularizations of a convex body}

\author{Natalia Tziotziou}

\date{}

\maketitle

\begin{abstract}
\footnotesize We establish new geometric inequalities comparing the volumes of sections and projections of a convex body, 
whose barycenter or Santal\'o point is at the origin, with those of its inner and outer regularizations. 
We also provide functional extensions of these inequalities to the setting of log-concave functions. 
Our approach relies on the recent optimal $M$-estimate of Bizeul and Klartag for isotropic convex bodies.
\end{abstract}

\section{Introduction}\label{section-1}

The purpose of this work is to establish new geometric inequalities comparing 
the volumes of sections and projections of a convex body---whose barycenter or 
Santal\'{o} point lies at the origin---with those of its inner and outer regularizations. 
Our method highlights the role of isotropic normalization as a unifying 
tool bridging symmetric and general convex bodies. 
These results extend naturally to the functional setting, 
yielding analogous inequalities for log-concave functions that satisfy 
volume-type bounds similar to those for convex bodies. 
A key ingredient in our approach is the recent optimal $M$-estimate 
of Bizeul and Klartag for isotropic convex bodies, which provides the 
quantitative foundation for our main results.

Throughout, let $K$ be a convex body in $\mathbb{R}^n$ with $0 \in {\rm int}(K)$. We define the \emph{outer} and \emph{inner regularizations} of $K$ by
\[
K_{\mathrm{out}} = {\rm conv}(K, -K)
\quad\text{and}\quad
K_{\mathrm{in}} = K \cap (-K).
\]
Clearly, $K_{\mathrm{in}} \subseteq K \subseteq K_{\mathrm{out}}$, and both $K_{\mathrm{out}}$ and $K_{\mathrm{in}}$ are origin-symmetric. We shall frequently use the duality relations
\begin{equation}\label{eq:duality}
(K_{\mathrm{out}})^{\circ} = (K^{\circ})_{\mathrm{in}}
\quad\text{and}\quad
(K_{\mathrm{in}})^{\circ} = (K^{\circ})_{\mathrm{out}},
\end{equation}
which follow from the identities
\[
({\rm conv}(A \cup B))^{\circ} = A^{\circ} \cap B^{\circ}
\quad\text{and}\quad
(A \cap B)^{\circ} = {\rm conv}(A^{\circ} \cup B^{\circ})
\]
valid for convex bodies $A,B$ with $0 \in {\rm int}(A \cap B)$. Here $A^{\circ}$ denotes the polar body of $A$.

We write ${\rm bar}(A)$ for the barycenter
\[
{\rm bar}(A) = \frac{1}{\vol_n(A)} \int_A x \, dx
\]
of a convex body $A \subset \mathbb{R}^n$, and $\vrad(A)$ for its \emph{volume radius},
\[
\vrad(A) = \left( \frac{\vol_n(A)}{\vol_n(B_2^n)} \right)^{1/n},
\]
where $B_2^n$ denotes the Euclidean unit ball and $\vol_n$ stands for $n$-dimensional volume.

\medskip

Our first result compares the volumes of the projections of $K_{\mathrm{out}}$ and $K_{\mathrm{in}}$ onto a $k$-dimensional subspace $H$ of $\mathbb{R}^n$.

\begin{theorem}\label{th:projections}
Let $K$ be a convex body in $\mathbb{R}^n$ with either ${\rm bar}(K) = 0$ or ${\rm bar}(K^{\circ}) = 0$. Then, for every $1 \ls k \ls n-1$ and any $H \in G_{n,k}$,
\[
\vrad(P_H(K_{\mathrm{out}})) \ls \frac{n g(n)}{k} \, \vrad(P_H(K_{\mathrm{in}})),
\]
where $g(n) \ls C (\ln n)^3$.
\end{theorem}

A related estimate was obtained by Vritsiou in \cite[Corollary~11]{Vritsiou-2024}, who proved that if ${\rm bar}(K) = 0$ then
\[
\vrad(P_H(K_{\mathrm{out}}))
\ls 
\left( \frac{n}{k} \right)^5 (\ln (en/k))^2 \, \vrad(P_H(K_{\mathrm{in}}))
\]
for all $1 \ls k \ls n-1$ and $H \in G_{n,k}$. Theorem~\ref{th:projections} thus yields an estimate that is nearly linear in $n/k$.

\medskip

Our second result concerns sections of $K_{\mathrm{out}}$ and $K_{\mathrm{in}}$ by a $k$-dimensional subspace $H \subset \mathbb{R}^n$.

\begin{theorem}\label{th:sections}
Let $K$ be a convex body in $\mathbb{R}^n$ with either ${\rm bar}(K) = 0$ or ${\rm bar}(K^{\circ})=0$. 
Then, for every $1 \ls k \ls n-1$ and any $H \in G_{n,k}$,
\[
\vrad(K_{\mathrm{out}} \cap H)
\ls 
\frac{n g(n)}{k} \, \vrad(K_{\mathrm{in}} \cap H),
\]
where $g(n) \ls C (\ln n)^3$.
\end{theorem}

It is instructive to compare Theorem~\ref{th:sections} with an inequality due to Rudelson \cite{Rudelson-2000a}, who showed that if $K$ is a convex body in $\mathbb{R}^n$, then for all $1 \ls k \ls n-1$ and all $H \in G_{n,k}$,
\[
\vrad((K - K) \cap H)
\ls 
c \, \min\{ n/k, \sqrt{k} \} \,
\max_{x\in \mathbb{R}^n} \vrad(K \cap (x + H)),
\]
where $c > 0$ is an absolute constant. Fradelizi further proved in \cite{Fradelizi-1997} that
\[
\max_{x \in \mathbb{R}^n} \vrad(K \cap (x + H))
\ls 
\frac{n+1}{k+1} \,
\vrad(K \cap ({\rm bar}(K) + H))^{1/k}.
\]
Combining these results yields
\begin{equation}\label{eq:rudelson-fradelizi-intro}
\vrad((K - K) \cap H)
\ls 
c \left( \frac{n}{k} \right)^2
\vrad(K \cap ({\rm bar}(K) + H))
\end{equation}
for all $1 \ls k \ls n-1$ and $H \in G_{n,k}$. Since $K_{\mathrm{in}} \subseteq K$ and $K - K \subseteq 2K_{\mathrm{out}}$, Theorem~\ref{th:sections} immediately implies a version of \eqref{eq:rudelson-fradelizi-intro} with improved dependence on $n/k$.

\medskip

Combining Theorems~\ref{th:projections} and \ref{th:sections} yields the following Blaschke–Santaló type inequality for projections of possibly non-symmetric convex bodies.

\begin{theorem}\label{th:BS-weak}
Let $K$ be a convex body in $\mathbb{R}^n$ such that either ${\rm bar}(K) = 0$ or ${\rm bar}(K^{\circ}) = 0$. Then, for every $1 \ls k \ls n-1$ and $H \in G_{n,k}$,
\[
\vrad(P_H(K)) \, \vrad(K^{\circ} \cap H)
\ls 
\frac{n g(n)}{k},
\]
where $g(n) \ls C (\ln n)^3$.
\end{theorem}

We include the proof in Section~\ref{section:5} to highlight that the three theorems above are closely related. A stronger version of Theorem~\ref{th:BS-weak} was proved by Vritsiou \cite{Vritsiou-2024}, who established that
\begin{equation}\label{eq:vritsiou}
\vrad(P_H(K)) \, \vrad(K^{\circ} \cap H)
\ls 
\frac{c n}{k},
\end{equation}
for some absolute constant $c > 0$. Moreover, \eqref{eq:vritsiou} is sharp: as shown in \cite{Vritsiou-2024}, if $S_n$ is a regular simplex of edge length $\sqrt{2}$ with barycenter at the origin, then the affine hull of any $k$ of its vertices, together with the average of the remaining $n + 1 - k$ vertices, forms a subspace $H_k \in G_{n,k}$ for which
\[
\vrad(P_{H_k}(S_n^{\circ})) \, \vrad(S_n \cap H_k) \approx \frac{n}{k}.
\]

For the proofs of the above theorems, we may assume without loss of generality that $K$ (or $K^{\circ}$) is isotropic (see 
below and in Section~\ref{section:2} for precise definitions and background). Under this assumption, stronger estimates can 
be obtained for a random subspace $H \in G_{n,k}$.

\begin{theorem}\label{th:random-subspace}
Let $K$ be an isotropic convex body in $\mathbb{R}^n$. Then, for any $1 \ls k \ls n-1$, a random subspace $H \in G_{n,k}$ satisfies
\begin{equation}\label{eq:iso-1}
\vrad(K_{\mathrm{out}} \cap H)
\ls 
\gamma_n \, \vrad(K_{\mathrm{in}} \cap H),
\end{equation}
with probability greater than $1 - e^{-k}$, where $\gamma_n \ls C (\ln n)^2$. Moreover, with the same probability, one also has
\begin{equation}\label{eq:iso-2}
\vrad(K_{\mathrm{out}} \cap H)
\ls 
\delta_n \, \vrad(K_{\mathrm{in}} \cap H),
\end{equation}
where $\delta_n \ls C \ln n$.
\end{theorem}

\medskip

We also establish functional analogues of Theorems~\ref{th:sections} and \ref{th:projections} for the class $\mathcal{L}^n$ of geometric log-concave integrable functions. These are the centered log-concave functions $f : \mathbb{R}^n \to [0, \infty)$ satisfying $f(0) = \|f\|_{\infty}$ and $0 < \int f < \infty$. For any $f \in \mathcal{L}^n$, define
\[
\Delta_{\mathrm{out}} f(x)
=
\sup \Big\{
\sqrt{f(x_1) f(-x_2)} : x = \tfrac{1}{2}(x_1 + x_2)
\Big\}
\quad\text{and}\quad
\Delta_{\mathrm{in}} f(x)
=
\min\{ f(x), f(-x) \}.
\]
If $f = \mathds{1}_K$ for some convex body $K$, then $\Delta_{\mathrm{out}} f = \mathds{1}_{\frac{1}{2}(K - K)}$ and $\Delta_{\mathrm{in}} f = \mathds{1}_{K \cap (-K)}$. Moreover, for every $x \in \mathbb{R}^n$, one has $\Delta_{\mathrm{in}} f(x) \ls f(x)$ and $\Delta_{\mathrm{in}} f(x) \ls \Delta_{\mathrm{out}} f(x)$. The following theorem extends Theorem~\ref{th:sections} to this functional setting.

\begin{theorem}\label{th:sections-functional}
For any $f \in \mathcal{L}^n$, any $1 \ls k \ls n-1$, and any $H \in G_{n,k}$,
\begin{equation}\label{eq:sections-functional-upper}
\left( \int_H \Delta_{\mathrm{out}} f(x) \, dx \right)^{1/k}
\ls 
C (n/k)^2 g(n)
\left( \int_H \Delta_{\mathrm{in}} f(x) \, dx \right)^{1/k},
\end{equation}
and
\begin{equation}\label{eq:sections-functional-lower}
c \left( \int_H f(x) \, dx \right)^{1/k}
\ls 
\left( \int_H \Delta_{\mathrm{out}} f(x) \, dx \right)^{1/k},
\end{equation}
where $c > 0$ and $C > 1$ are absolute constants.
\end{theorem}

Given a nonnegative measurable function $g : \mathbb{R}^n \to [0, \infty)$ and $H \in G_{n,k}$, the orthogonal projection 
of $g$ onto $H$ is defined by
\[
(P_H g)(z) = \sup\{ g(y + z) : y \in H^{\perp} \}.
\]
With this definition, we obtain the following functional analogue of Theorem~\ref{th:projections}.

\begin{theorem}\label{th:projections-functional}
For any $f \in \mathcal{L}^n$, any $1 \ls k \ls n-1$, and any $H \in G_{n,k}$,
\begin{equation}\label{eq:projections-functional}
\left( \int_H P_H(\Delta_{\mathrm{out}} f)(x) \, dx \right)^{1/k}
\ls 
C (n/k)^2 g(n)
\left( \int_H P_H(\Delta_{\mathrm{in}} f)(x) \, dx \right)^{1/k},
\end{equation}
where $C > 0$ is an absolute constant.
\end{theorem}

To prove Theorems~\ref{th:sections-functional} and \ref{th:projections-functional}, we employ the family of K.~Ball's bodies $K_p(f)$ associated with a log-concave function $f$, allowing us to transfer the problem to the setting of convex bodies and apply Theorems~\ref{th:sections} and \ref{th:projections} accordingly.

\medskip

The foundation of our results lies in recent advances concerning isotropic convex bodies and, in particular, in the sharp $MM^{\ast}$-estimate. We denote by $p_K(x) = \inf\{ t > 0 : x \in tK \}$ the Minkowski functional of $K$, and by $h_K(y) = \max\{ \langle x, y \rangle : x \in K \}$ its support function. The parameters
\[
M(K) = \int_{S^{n-1}} p_K(x) \, d\sigma(x)
\quad\text{and}\quad
M^{\ast}(K) = \int_{S^{n-1}} h_K(x) \, d\sigma(x),
\]
where $\sigma$ denotes the rotationally invariant probability measure on the unit sphere $S^{n-1}$, play a central role in the asymptotic theory of finite-dimensional normed spaces. It is well known that $M(K) M^{\ast}(K) \gr 1$ for every convex body $K$ with $0 \in {\rm int}(K)$.

A classical result of Figiel–Tomczak-Jaegermann \cite{Figiel-Tomczak-1979}, Lewis \cite{Lewis-1979}, and Pisier \cite{Pisier-1982} asserts that for any symmetric convex body $K \subset \mathbb{R}^n$ there exists $T \in GL_n$ such that
\[
M(TK) M^{\ast}(TK) \ls c \ln n,
\]
where $c > 0$ is an absolute constant. Without assuming symmetry, it is natural to consider
\[
E(K) = \inf M(TK) M^{\ast}(TK),
\]
where the infimum runs over all invertible affine transformations $T$ of $\mathbb{R}^n$ satisfying $0 \in {\rm int}(TK)$. It turns out that the isotropic position yields a sharp upper bound for $E(K)$.

A convex body $K \subset \mathbb{R}^n$ is called \emph{isotropic} if it has volume $1$, its barycenter is at the origin, and there exists a constant $L_K > 0$ such that
\[
\int_K \langle x, \xi \rangle^2 \, dx = L_K^2
\quad\text{for all } \xi \in S^{n-1}.
\]
Every convex body $K$ admits an isotropic affine image, unique up to orthogonal transformations (see \cite{Milman-Pajor-LK}). Using this canonical position, one defines the isotropic constant $L_K$, an affine invariant of $K$.

A central question in asymptotic convex geometry, posed by Bourgain \cite{Bourgain-1986}, asked whether there exists an absolute constant $C > 0$ such that
\[
L_n := \max\{ L_K : K \text{ isotropic convex body in } \mathbb{R}^n \} \ls C
\]
for all $n \gr 1$. This was recently resolved affirmatively by Klartag and Lehec \cite{KL}, following major progress by Guan \cite{Guan}, and soon afterward an alternative proof was given by Bizeul \cite{Bizeul-2025}.

E.~Milman proved in \cite{EMilman-2014} that if $K$ is isotropic in ${\mathbb R}^n$, then
\begin{equation}\label{eq:EM}
M^{\ast}(K) \ls C \sqrt{n} (\ln n)^2 L_K \ls c_1 \sqrt{n} (\ln n)^2,
\end{equation}
where the second inequality uses the boundedness of $L_n$. The dependence on $n$ is optimal up to the logarithmic term.  

The dual estimate for $M(K)$ in the isotropic position was obtained recently by Bizeul and Klartag \cite{Bizeul-Klartag-2025}, who showed that
\begin{equation}\label{eq:BK}
M(K) \ls c_2 \frac{\log n}{\sqrt{n}}.
\end{equation}

Combining \eqref{eq:EM} and \eqref{eq:BK} gives
\begin{equation}\label{eq:solution}
E(K) \ls c_3 (\ln n)^3
\end{equation}
for every convex body $K \subset \mathbb{R}^n$.  
Since $M(K) M^{\ast}(K) \gr 1$ holds universally, \eqref{eq:solution} provides a sharp upper bound for $E(K)$ up to a factor of $(\ln n)^3$.

In Section~\ref{section:2} we review these results in more detail, and in Section~\ref{section:3} we collect classical geometric inequalities relevant to our setting.

Our approach exploits these refined $M$- and $M^{\ast}$-estimates to derive strong regularity estimates for covering numbers of isotropic convex bodies. These, in turn, allow a careful comparison of sections and projections of the inner and outer symmetrizations $K_{\mathrm{in}}$ and $K_{\mathrm{out}}$, leading directly to Theorems~\ref{th:projections} and \ref{th:sections}. In particular, the isotropic position serves not only as a convenient normalization but also as a tool for transferring results from the symmetric to the general, non-symmetric setting. This strategy underpins the functional extensions presented in Theorems~\ref{th:sections-functional} and \ref{th:projections-functional}, showing that the geometric inequalities naturally extend to the broader class of log-concave functions.

\section{Notation and background}\label{section:2}

We work in $\mathbb{R}^n$, equipped with the standard inner product $\langle \cdot, \cdot \rangle$.  
The corresponding Euclidean norm is denoted by $|\cdot|$, the Euclidean unit ball by $B_2^n$, and the unit sphere by $S^{n-1}$.  
Volume in $\mathbb{R}^n$ is denoted by $\vol_n$, while $\omega_n = \vol_n(B_2^n)$ stands for the volume of the Euclidean unit ball.  

We denote by $\sigma$ the rotationally invariant probability measure on $S^{n-1}$ and by $\nu$ the Haar measure on $O(n)$.  
The Grassmann manifold $G_{n,k}$ of $k$-dimensional subspaces of $\mathbb{R}^n$ is equipped with the Haar probability measure $\nu_{n,k}$.  
For each integer $1 \ls k \ls n-1$ and every $H \in G_{n,k}$, we denote by $P_H$ the orthogonal projection from $\mathbb{R}^n$ onto $H$, and set
\[
B_H := B_2^n \cap H, \qquad S_H := S^{n-1} \cap H.
\]

\medskip
\noindent
The letters $c, c', c_1, c_2, \ldots$ denote absolute positive constants whose value may change from line to line.  
Whenever we write $a \approx b$, we mean that there exist absolute constants $c_1, c_2 > 0$ such that $c_1 a \ls b \ls c_2 a$.  
Similarly, for convex bodies $K, C \subseteq \mathbb{R}^n$, we write $K \approx C$ if there exist absolute constants $c_1, c_2 > 0$ such that
$c_1 K \subseteq C \subseteq c_2 K$.

\medskip
A \emph{convex body} in $\mathbb{R}^n$ is a compact convex subset $K$ with nonempty interior.  
We say that $K$ is \emph{symmetric} if $K = -K$, and \emph{centered} if its barycenter ${\rm bar}(K)$ is at the origin.

The \emph{radial function} of a convex body $K$ with $0 \in {\rm int}(K)$ is
\[
\rho_K(x) = \max\{ t > 0 : t x \in K \}, \qquad x \in \mathbb{R}^n \setminus \{0\},
\]
and the \emph{support function} of $K$ is defined for $y \in \mathbb{R}^n$ by
\[
h_K(y) = \max \{ \langle x, y \rangle : x \in K \}.
\]
The \emph{radius} of $K$ is $R(K) = \max\{ |x| : x \in K \}$, and the \emph{volume radius} is
\[
{\rm vrad}(K) = \left( \frac{\vol_n(K)}{\vol_n(B_2^n)} \right)^{1/n}.
\]
The \emph{polar body} $K^{\circ}$ of a convex body $K$ with $0 \in {\rm int}(K)$ is given by
\[
K^{\circ} = \{ x \in \mathbb{R}^n : \langle x, y \rangle \ls 1 \text{ for all } y \in K \}.
\]

\medskip
An absolutely continuous Borel probability measure $\mu$ on $\mathbb{R}^n$ is called \emph{log-concave} if its density $f_\mu$ is of the form
$f_\mu = e^{-\varphi}$ with $\varphi : \mathbb{R}^n \to \mathbb{R} \cup \{+\infty\}$ convex.  
The uniform probability measure on any convex body is log-concave.

The barycenter of $\mu$ is
\[
{\rm bar}(\mu) := \int_{\mathbb{R}^n} x f_\mu(x) \, dx,
\]
and its \emph{isotropic constant} is the affine-invariant quantity
\begin{equation}\label{definition-isotropic}
L_\mu := \left( \frac{\|f_\mu\|_{\infty}}{\int_{\mathbb{R}^n} f_\mu(x) \, dx} \right)^{\frac{1}{n}}
\det({\rm Cov}(\mu))^{\frac{1}{2n}},
\end{equation}
where
\[
{\rm Cov}(\mu) := \int x \otimes x \, d\mu(x)
- \left( \int x \, d\mu(x) \right) \otimes \left( \int x \, d\mu(x) \right)
\]
is the covariance matrix of $\mu$.  
A log-concave probability measure $\mu$ on $\mathbb{R}^n$ is called \emph{isotropic} if ${\rm bar}(\mu) = 0$ and ${\rm Cov}(\mu) = I_n$.  
A convex body $K$ of volume $1$ is isotropic if and only if the log-concave probability measure with density $L_K^n \mathds{1}_{K/L_K}$ is isotropic.

\medskip
K.~Ball~\cite{Ball-1988} showed that, in every dimension $n$,
\[
\sup_\mu L_\mu \ls C \sup_K L_K,
\]
where the suprema are taken over all log-concave probability measures $\mu$ and all convex bodies $K \subseteq \mathbb{R}^n$, respectively.
Around 1985--86 (published in 1990), Bourgain~\cite{Bourgain-1991} obtained the bound $L_n \ls c n^{1/4} \ln n$, later improved by Klartag~\cite{Klartag-2006} to $L_n \ls c n^{1/4}$.  
These estimates remained the best known until 2020.  
In a breakthrough, Chen~\cite{C} proved that for every $\varepsilon > 0$, one has $L_n \ls n^{\varepsilon}$ for all sufficiently large $n$.  
This development initiated a series of works culminating in the final affirmative solution of Bourgain's problem 
by Klartag and Lehec~\cite{KL}, following an important contribution by Guan~\cite{Guan}.  
Shortly thereafter, Bizeul~\cite{Bizeul-2025} provided another proof of the conjecture.

\medskip
The study of $E(K)$ in the non-symmetric case began with the work of Banaszczyk, Litvak, Pajor, and Szarek~\cite{Banaszczyk-Litvak-Pajor-Szarek-1999}, who showed that if $K$ is a convex body in ${\mathbb R}^n$ in John’s position (i.e., its maximal-volume inscribed ellipsoid is the Euclidean unit ball), then
\[
M^{\ast}(K) \ls c \sqrt{n} \sqrt{\ln n}.
\]
Since $K \supseteq B_2^n$ in John’s position, we also have the trivial bound $M(K) \ls M(B_2^n) = 1$, hence $E(K) \ls c \sqrt{n} \sqrt{\ln n}$.
Rudelson~\cite{Rudelson-2000b} improved this to
\[
E(K) \ls c n^{1/3} (\ln n)^b
\]
for some absolute constant $b > 0$.  
This remained the best known bound until the recent work of Bizeul and Klartag.

After earlier estimates of order $n^{3/4} L_K$ for $M^{\ast}(K)$ in the isotropic position (see~\cite[Chapter~9]{BGVV-book}),  
E.~Milman~\cite{EMilman-2014} proved that if $K$ is a symmetric isotropic convex body in $\mathbb{R}^n$, then
\[
M^{\ast}(K) \ls c_1 \sqrt{n} (\ln n)^2 L_K,
\]
and the same bound extends to non-symmetric isotropic convex bodies.  
For the dual problem, estimating $M(K)$ in isotropic position, the first nontrivial results appeared in~\cite{Giannopoulos-Stavrakakis-Tsolomitis-Vritsiou-TAMS}.  
The best known bound in the symmetric case, due to Giannopoulos and E.~Milman~\cite{Giannopoulos-EMilman-2014}, was
\[
M(K) \ls \frac{C (n \ln n)^{1/3}}{\sqrt{n}},
\]
while in the non-symmetric case, Vritsiou~\cite{Vritsiou-2024} obtained the estimate 
\[
M(K) \ls \frac{c n^{5/11} (\ln n)^{5/22}}{\sqrt{n}}.
\]

\medskip
For background on isotropic convex bodies and log-concave measures, we refer to~\cite{BGVV-book};  
for general information on the local theory of normed spaces, see~\cite{AGA-book, AGA-book-2, Pisier-book}.

\section{Geometric inequalities}\label{section:3}

In this section, we review several classical geometric inequalities that will be used in the sequel, beginning with the Blaschke--Santal\'{o} inequality.  

Let $K$ be a convex body in $\mathbb{R}^n$. The function $\vol_n(K)\,\vol_n((K-z)^{\circ})$,
defined on $\operatorname{int}(K)$, is strictly convex and attains a unique minimum at the \emph{Santal\'{o} point} $s(K)$ of $K$. The Blaschke--Santal\'{o} inequality asserts that
\[
\vol_n(K)\,\vol_n((K-s(K))^{\circ}) \ls \omega_n^2.
\]
Moreover, if $\operatorname{bar}(K)=0$, then $s(K^{\circ})=0$. Since $(K^{\circ})^{\circ}=K$, we obtain
\[
\vol_n(K)\,\vol_n(K^{\circ})
= \vol_n((K^{\circ}-s(K^{\circ}))^{\circ})\,\vol_n(K^{\circ})
\ls \omega_n^2.
\]
Hence we have the following classical result.

\begin{theorem}[Blaschke--Santal\'{o}]\label{th:B-S}
Let $K$ be a convex body in $\mathbb{R}^n$ such that either $\operatorname{bar}(K)=0$ or $s(K)=0$. Then,
\[
\vol_n(K)\,\vol_n(K^{\circ}) \ls \omega_n^2.
\]
\end{theorem}

A refinement of this inequality was given by Meyer and Pajor~\cite{Meyer-Pajor-1990}.  
For $\lambda \in (0,1)$, a hyperplane
\[
F = \{x \in \mathbb{R}^n : \langle x, u_F \rangle = \alpha_F\},
\]
where $u_F \in \mathbb{R}^n \setminus \{0\}$ and $\alpha_F \in \mathbb{R}$, is said to be \emph{$\lambda$-separating} for $K$ if
\[
\vol_n(\{x \in K : \langle x, u_F \rangle \gr \alpha_F\}) = \lambda\,\vol_n(K).
\]
Note that a $\lambda$-separating hyperplane necessarily intersects the interior of $K$.

\begin{theorem}[Meyer--Pajor]\label{th:meyer-pajor}
Let $K$ be a convex body in $\mathbb{R}^n$ and $F$ a $\lambda$-separating hyperplane for $K$, where $\lambda \in (0,1)$. Then, there exists $z \in \operatorname{int}(K) \cap F$ such that
\[
\vol_n(K)\,\vol_n((K-z)^{\circ}) \ls \frac{\omega_n^2}{4\lambda(1-\lambda)}.
\]
Moreover, $z$ is the unique point in $\operatorname{int}(K) \cap F$ such that
\[
\operatorname{bar}((K-z)^{\circ}) \in \{t u_F : t \in \mathbb{R}\}.
\]
\end{theorem}

In the opposite direction, the Bourgain--Milman inequality~\cite{BM} provides a universal lower bound.

\begin{theorem}[Bourgain--V.~Milman]\label{th:B-M}
Let $K$ be a convex body in $\mathbb{R}^n$ with $0 \in \operatorname{int}(K)$. Then,
\[
\vol_n(K)\,\vol_n(K^{\circ})
\gr \vol_n(K)\,\vol_n((K-s(K))^{\circ})
\gr c^n \omega_n^2,
\]
where $c>0$ is an absolute constant.
\end{theorem}

Classical results of Rogers--Shephard~\cite{Rogers-Shephard-1957} and V.~Milman--Pajor~\cite{VMilman-Pajor-2000} compare the volume of a convex body with those of its inner and outer regularizations.

\begin{theorem}[Rogers--Shephard / V.~Milman--Pajor]\label{th:R-S}
Let $K$ be a convex body in $\mathbb{R}^n$. Then,
\[
2^n \vol_n(K) \ls \vol_n(K-K) \ls \binom{2n}{n}\vol_n(K) \ls 4^n \vol_n(K).
\]
Moreover, if $0 \in \operatorname{int}(K)$, then
\[
\vol_n(K_{\mathrm{out}}) = \vol_n(\operatorname{conv}(K, -K)) \ls 2^n \vol_n(K),
\]
and if $\operatorname{bar}(K)=0$, then
\[
\vol_n(K_{\mathrm{in}}) = \vol_n(K \cap (-K)) \gr 2^{-n} \vol_n(K).
\]
\end{theorem}

If $s(K)=0$, then $\operatorname{bar}(K^{\circ})=0$, and since $(K_{\mathrm{in}})^{\circ} = (K^{\circ})_{\mathrm{out}}$, the Bourgain--Milman inequality yields
\[
2^n \vol_n(K_{\mathrm{in}})\,\vol_n(K^{\circ})
\gr \vol_n(K_{\mathrm{in}})\,\vol_n((K^{\circ})_{\mathrm{out}})
\gr c^n \omega_n^2
\gr c^n \vol_n(K)\,\vol_n(K^{\circ}),
\]
and hence
\[
\vol_n(K_{\mathrm{in}}) \gr (c/2)^n \vol_n(K),
\]
an observation due to Rudelson~\cite{Rudelson-2000b}.

The following inequality estimates the product of the volumes of a projection of a convex body and the corresponding orthogonal section 
(see \cite{Rogers-Shephard-1958} for the first and \cite{Spingarn-1993} for the second claim).

\begin{theorem}[Rogers--Shephard / Spingarn]\label{th:R-S-S}
Let $K$ be a convex body in $\mathbb{R}^n$ with $0 \in \operatorname{int}(K)$. Then, for every $1 \ls k \ls n-1$ and any $H \in G_{n,k}$,
\[
\vol_k(P_H(K))\,\vol_{n-k}(K \cap H^{\perp}) \ls \binom{n}{k}\,\vol_n(K).
\]
If $\operatorname{bar}(K)=0$, then also
\[
\vol_n(K) \ls \vol_k(P_H(K))\,\vol_{n-k}(K \cap H^{\perp}).
\]
\end{theorem}

Rudelson's inequality~\cite{Rudelson-2000a} compares the volume of a central section of $K-K$ with that of the maximal corresponding section of $K$.

\begin{theorem}[Rudelson]\label{th:rudelson}
Let $K$ be a convex body in $\mathbb{R}^n$. Then, for every $1 \ls k \ls n-1$ and any $H \in G_{n,k}$,
\[
\vol_k((K-K)\cap H)^{1/k} \ls c \min\left\{ \sqrt{k}, \frac{n}{k} \right\} \max_{x \in \mathbb{R}^n} \vol_k(K \cap (x+H))^{1/k}.
\]
\end{theorem}

An inequality of Fradelizi~\cite{Fradelizi-1997} compares the maximal section of a convex body with the section passing through its barycenter.

\begin{theorem}[Fradelizi]\label{th:fradelizi}
Let $K$ be a convex body in $\mathbb{R}^n$. Then, for every $1 \ls k \ls n-1$ and any $H \in G_{n,k}$,
\[
\max_{x \in \mathbb{R}^n} \vol_k(K \cap (x+H))^{1/k}
\ls \frac{n+1}{k+1}\,\vol_k(K \cap (\operatorname{bar}(K)+H))^{1/k}.
\]
\end{theorem}

Combining these two results yields
\begin{equation}\label{eq:rudelson-fradelizi}
\vol_k((K-K)\cap H)^{1/k}
\ls c\left( \frac{n}{k} \right)^2 \vol_k(K \cap (\operatorname{bar}(K)+H))^{1/k}
\end{equation}
for every $1 \ls k \ls n-1$ and $H \in G_{n,k}$.
A direct proof of~\eqref{eq:rudelson-fradelizi} with improved dependence on $n/k$ will be given in Theorem~\ref{th:sections} (Section~\ref{section:5}).

Let $K$ be a convex body in $\mathbb{R}^n$ with $0 \in \operatorname{int}(K)$.  
Recall that for $1 \ls k \ls n-1$ and $H \in G_{n,k}$, the projection $P_H(K)$ is the polar body of $K^{\circ} \cap H$ in $H$.  
If $K$ is symmetric, then $P_H(K)$ and $K^{\circ} \cap H$ are polar symmetric convex bodies in $H$, so that
\[
\vrad(P_H(K))\,\vrad(K^{\circ}\cap H) \ls 1
\]
by the Blaschke--Santal\'{o} inequality.  
Vritsiou~\cite{Vritsiou-2024} extended this to non-symmetric convex bodies.

\begin{theorem}[Vritsiou]\label{th:vritsiou}
Let $K$ be a convex body in $\mathbb{R}^n$ such that either $\operatorname{bar}(K)=0$ or $s(K)=0$. Then, for every $1 \ls k \ls n-1$ and any $H \in G_{n,k}$,
\[
\vrad(P_H(K))\,\vrad(K^{\circ}\cap H) \ls c\,\frac{n}{k},
\]
where $c>0$ is an absolute constant.
\end{theorem}

The proof of Theorem~\ref{th:vritsiou} combines the Meyer--Pajor refinement of the Blaschke--Santal\'{o} inequality (Theorem~\ref{th:meyer-pajor}) with Gr\"{u}nbaum-type inequalities due to Stephen, Zhang, and Myroshnychenko~\cite{Stephen-Zhang-2017, M-Stephen-Zhang-2018}.

\section{Covering numbers of isotropic convex bodies}\label{section:4}

Recall that if $A$ and $B$ are two convex bodies in $\mathbb{R}^n$, the \emph{covering number} $N(A,B)$ of $A$ by $B$ is the least
integer $N$ for which there exist $N$ translates of $B$ whose union covers $A$:
\begin{equation*}
N(A,B)
   = \min\Big\{ N \in \mathbb{N} :
      \exists\, x_1, \ldots, x_N \in \mathbb{R}^n
      \ \text{such that} \
      A \subseteq \bigcup_{j=1}^N (x_j + B) \Big\}.
\end{equation*}

The Sudakov and dual Sudakov inequalities (see~\cite[Chapter~4]{AGA-book}) state that if $K$ is a symmetric convex body in
$\mathbb{R}^n$, then for every $t > 0$,
\begin{equation*}
N(K, t B_2^n) \ls \exp\!\big( c n M^{\ast}(K)^2 / t^2 \big)
\quad \text{and} \quad
N(B_2^n, t K) \ls \exp\!\big( c n M(K)^2 / t^2 \big),
\end{equation*}
where $c > 0$ is an absolute constant.

\vspace{0.8em}
\noindent
V.~Milman proved in~\cite{VMilman-1986} that there exists an absolute constant $\beta > 0$ such that every centered convex body $A \subset \mathbb{R}^n$ admits a linear image $\tilde{A}$ satisfying $\vol_n(\tilde{A}) = \vol_n(B_2^n)$ and
\begin{equation}\label{betaMposition}
\max\Bigl\{
   N(\tilde{A}, B_2^n),
   N(B_2^n, \tilde{A}),
   N(\tilde{A}^{\circ}, B_2^n),
   N(B_2^n, \tilde{A}^{\circ})
   \Bigr\}
   \ls \exp(\beta n).
\end{equation}
A convex body $A$ satisfying this estimate is said to be in \emph{$M$-position} with constant~$\beta$.

Pisier~\cite{Pisier-1989} proposed a different approach, allowing one to construct an entire family of $M$-positions and to obtain
more precise quantitative information on the corresponding covering numbers.

\begin{theorem}[Pisier]\label{th:pisier-alpha-regular}
For every $0 < \alpha < 2$ and every symmetric convex body $A \subset \mathbb{R}^n$, there exists
a linear image $\tilde{A}$ of $A$ such that
\[
\max\Bigl\{
   N(\tilde{A}, t B_2^n),
   N(B_2^n, t \tilde{A}),
   N(\tilde{A}^{\circ}, t B_2^n),
   N(B_2^n, t \tilde{A}^{\circ})
   \Bigr\}
   \ls \exp\!\big( c(\alpha) n / t^{\alpha} \big)
\]
for every $t \gr c(\alpha)^{1/\alpha}$, where $c(\alpha)$ depends only on~$\alpha$ and satisfies
$c(\alpha) = O\!\big( (2 - \alpha)^{-\alpha/2} \big)$ as $\alpha \to 2^-$.
\end{theorem}

A convex body $A$ satisfying the estimate in Theorem~\ref{th:pisier-alpha-regular} is said to be in \emph{$\alpha$-regular $M$-position} with constant~$c(\alpha)$.

\vspace{0.8em}
\noindent
The following proposition, a simple consequence of the $M^{\ast}$ estimate~\eqref{eq:EM} and the $M$ estimate~\eqref{eq:BK},
shows that every isotropic convex body $K \subset \mathbb{R}^n$ is (almost) in $2$-regular $M$-position.
Below, we set $r_n = \omega_n^{-1/n}$; note that $r_n \approx \sqrt{n}$.

\begin{proposition}\label{prop:M-position}
Let $K$ be an isotropic convex body in $\mathbb{R}^n$. Then, for every $t > 0$,
\begin{align}
\max\Bigl\{ N(K, t r_n B_2^n),\; N(B_2^n, t r_n K^{\circ}) \Bigr\}
   &\ls \exp\!\Big( \frac{\gamma_n^2 n}{t^2} \Big), \label{eq:M-position-a} \\
\max\Bigl\{ N(r_n B_2^n, t K),\; N(r_n K^{\circ}, t B_2^n) \Bigr\}
   &\ls \exp\!\Big( \frac{\delta_n^2 n}{t^2} \Big), \label{eq:M-position-b}
\end{align}
where $\gamma_n \ls c_1 (\ln n)^2$ and $\delta_n \ls c_2 \ln n$.
\end{proposition}

\begin{proof}
Let $K_{\rm in}$ and $K_{\rm out}$ denote the inner and outer regularizations of $K$. Observe that
\begin{align}
M(K_{\rm in})
   &= M^{\ast}\big( (K^{\circ})_{\rm out} \big)
   \ls M^{\ast}(K^{\circ} - K^{\circ})
   = 2 M^{\ast}(K^{\circ})
   = 2 M(K), \label{eq:M-1} \\
M^{\ast}(K_{\rm out})
   &\ls M^{\ast}(K - K)
   = 2 M^{\ast}(K). \label{eq:M-2}
\end{align}
Since $K_{\rm in} \subseteq K \subseteq K_{\rm out}$, it is clear that
$M(K_{\rm out}) \ls M(K)$ and $M^{\ast}(K_{\rm in}) \ls M^{\ast}(K)$.

The Sudakov inequality, combined with~\eqref{eq:M-2}, gives
\begin{align}\label{eq:M-3}
N(K, t r_n B_2^n)
   &\ls N(K_{\rm out}, t r_n B_2^n)
   \ls \exp\!\big( c n M^{\ast}(K_{\rm out})^2 / (r_n t)^2 \big) \\
\nonumber
   &\ls \exp\!\big( 4 c n M^{\ast}(K)^2 / (r_n t)^2 \big)
   \ls \exp\!\big( \gamma_n^2 n / t^2 \big).
\end{align}

Similarly, using~\eqref{eq:M-1} we obtain
\begin{align}\label{eq:M-4}
N(r_n K^{\circ}, t B_2^n)
   &\ls N(r_n (K_{\rm in})^{\circ}, t B_2^n)
   \ls \exp\!\big( c n r_n^2 M^{\ast}((K_{\rm in})^{\circ})^2 / t^2 \big) \\
\nonumber
   &= \exp\!\big( c n r_n^2 M(K_{\rm in})^2 / t^2 \big)
   \ls \exp\!\big( 4 c n r_n^2 M(K)^2 / t^2 \big)
   \ls \exp\!\big( \delta_n^2 n / t^2 \big).
\end{align}

Applying the dual Sudakov inequality and~\eqref{eq:M-1} yields
\begin{align}\label{eq:M-5}
N(r_n B_2^n, t K)
   &\ls N(r_n B_2^n, t K_{\rm in})
   \ls \exp\!\big( c n r_n^2 M(K_{\rm in})^2 / t^2 \big) \\
\nonumber
   &\ls \exp\!\big( 4 c n r_n^2 M(K)^2 / t^2 \big)
   \ls \exp\!\big( \delta_n^2 n / t^2 \big),
\end{align}
and, using~\eqref{eq:M-2}, we similarly have
\begin{align}\label{eq:M-6}
N(B_2^n, t r_n K^{\circ})
   &\ls N(B_2^n, t r_n (K_{\rm out})^{\circ})
   \ls \exp\!\big( c n M((K_{\rm out})^{\circ})^2 / (r_n t)^2 \big) \\
\nonumber
   &\ls \exp\!\big( c n M^{\ast}(K_{\rm out})^2 / (r_n t)^2 \big)
   = \exp\!\big( 4 c n M^{\ast}(K)^2 / (r_n t)^2 \big)
   \ls \exp\!\big( \gamma_n^2 n / t^2 \big).
\end{align}

Combining~\eqref{eq:M-3}--\eqref{eq:M-6} completes the proof.
\end{proof}

\section{Projections and sections of the outer and inner regularizations}\label{section:5}

In this section we compare the volumes of sections and projections of a convex body with those of its inner and outer 
regularizations. We begin with the proof of Theorem~\ref{th:projections}.

\begin{proof}[Proof of Theorem~$\ref{th:projections}$]
Assume first that ${\rm bar}(K) = 0$. Since $T(K_{\rm in}) = (TK)_{\rm in}$ and $T(K_{\rm out}) = (TK)_{\rm out}$ for every $T \in GL_n$, we may assume that $K$ is isotropic. 
From \eqref{eq:M-3} we know that $N_t = N(K_{\rm out}, t r_n B_2^n) \ls \exp(\gamma_n^2 n / t^2)$ for every $t > 0$. Thus, there exist
$x_1, \ldots, x_{N_t}$ such that 
\[
K_{\rm out} \subseteq \bigcup_{i=1}^{N_t} (x_i + t r_n B_2^n).
\]
Projecting onto a $k$-dimensional subspace $H$, we obtain
\[
P_H(K_{\rm out}) \subseteq \bigcup_{i=1}^{N_t} (P_H(x_i) + t r_n B_H),
\]
where $B_H$ denotes the Euclidean unit ball in $H$. Hence,
\[
\vol_k(P_H(K_{\rm out})) \ls N_t \, \vol_k(t r_n B_H) 
   \ls \exp(\gamma_n^2 n / t^2) (t r_n)^k \omega_k.
\]
Choosing $t = \gamma_n \sqrt{n/k}$ minimizes the right-hand side and gives
\[
\vol_k(P_H(K_{\rm out})) \ls e^k r_n^k \gamma_n^k (n/k)^{k/2} \omega_k,
\]
and therefore,
\begin{equation}\label{eq:P-1}
\vrad(P_H(K_{\rm out})) \ls e r_n \gamma_n \sqrt{n/k}.
\end{equation}

On the other hand, \eqref{eq:M-5} shows that 
$N'_t = N(r_n B_2^n, t K_{\rm in}) \ls \exp(\delta_n^2 n / t^2)$ for every $t > 0$. 
Proceeding in the same way, we obtain
\[
\vol_k(r_n B_H) \ls N'_t \, \vol_k(t P_H(K_{\rm in})) 
   \ls \exp(\delta_n^2 n / t^2) t^k \vol_k(P_H(K_{\rm in})).
\]
Choosing $t = \delta_n \sqrt{n/k}$, we get
\[
r_n^k \vol_k(B_H) \ls e^k \delta_n^k (n/k)^{k/2} \vol_k(P_H(K_{\rm in})),
\]
and hence,
\begin{equation}\label{eq:P-2}
r_n \ls e \delta_n \sqrt{n/k} \, \vrad(P_H(K_{\rm in})).
\end{equation}
Combining \eqref{eq:P-1} and \eqref{eq:P-2}, we find that
\begin{equation}\label{eq:P-b}
\vrad(P_H(K_{\rm out})) \ls e^2 \delta_n \gamma_n (n/k) \vrad(P_H(K_{\rm in})).
\end{equation}
The first claim of the theorem follows with $g(n) = e^2 \delta_n \gamma_n \ls c (\ln n)^3$.

Next, assume that ${\rm bar}(K^{\circ}) = 0$, in which case $s(K) = 0$. 
We may assume that $K^{\circ}$ is isotropic. 
Replacing $K$ by $K^{\circ}$ in \eqref{eq:M-4} and using \eqref{eq:duality}, we get
\[
N(r_n K_{\rm out}, t B_2^n) \ls \exp(\delta_n^2 n / t^2).
\]
Applying the argument of the first part of the proof, we obtain
\begin{equation}\label{eq:P-3}
r_n \vrad(P_H(K_{\rm out})) \ls e \delta_n \sqrt{n/k}.
\end{equation}
Replacing $K$ by $K^{\circ}$ in \eqref{eq:M-6} and using again \eqref{eq:duality}, we find that 
$N(B_2^n, t r_n K_{\rm in}) \ls \exp(\gamma_n^2 n / t^2)$, and by the same reasoning,
\begin{equation}\label{eq:P-4}
1 \ls e r_n \gamma_n \sqrt{n/k} \, \vrad(P_H(K_{\rm in})).
\end{equation}
Combining \eqref{eq:P-3} and \eqref{eq:P-4}, we obtain
\begin{equation}\label{eq:P-s}
\vrad(P_H(K_{\rm out})) \ls e^2 \delta_n \gamma_n (n/k) \vrad(P_H(K_{\rm in})),
\end{equation}
and the second claim of the theorem follows with $g(n) = e^2 \delta_n \gamma_n \ls c (\ln n)^3$.
\end{proof}

Theorem~\ref{th:sections} follows easily from Theorem~\ref{th:projections}.

\begin{proof}[Proof of Theorem~$\ref{th:sections}$]
Since $P_H((K^{\circ})_{\rm in})$ is symmetric and its polar body in $H$ is $K_{\rm out} \cap H$, 
the Blaschke--Santal\'{o} inequality implies
\begin{equation}\label{eq:S-1}
\vrad(K_{\rm out} \cap H) \, \vrad(P_H((K^{\circ})_{\rm in})) \ls 1.
\end{equation}
Similarly, since $P_H((K^{\circ})_{\rm out})$ is symmetric and its polar body in $H$ is $K_{\rm in} \cap H$, 
the Bourgain--Milman inequality gives
\begin{equation}\label{eq:S-2}
c_1 \ls \vrad(K_{\rm in} \cap H) \, \vrad(P_H((K^{\circ})_{\rm out})).
\end{equation}
Assuming that either ${\rm bar}(K) = 0$ or ${\rm bar}(K^{\circ})=0$, we may apply Theorem~\ref{th:projections} to $K^{\circ}$ and obtain
\[
\vrad(P_H((K^{\circ})_{\rm out})) \ls c_2 \frac{n g(n)}{k} \, \vrad(P_H((K^{\circ})_{\rm in})).
\]
Combining this estimate with \eqref{eq:S-1} and \eqref{eq:S-2}, we get
\begin{align*}
c_1 
&\ls \vrad(K_{\rm in} \cap H) \, \vrad(P_H((K^{\circ})_{\rm out})) \\
&\ls c_2 \frac{n g(n)}{k} \, \vrad(K_{\rm in} \cap H) \, \vrad(P_H((K^{\circ})_{\rm in})) \\
&\ls c_2 \frac{n g(n)}{k} \, \vrad(K_{\rm in} \cap H) \, \vrad(K_{\rm out} \cap H)^{-1}.
\end{align*}
Hence,
\[
\vrad(K_{\rm out} \cap H) \ls c_3 \frac{n g(n)}{k} \, \vrad(K_{\rm in} \cap H),
\]
where $c_3 = c_2 / c_1$. 
\end{proof}

We now turn to the proof of Theorem~\ref{th:BS-weak}.

\begin{proof}[Proof of Theorem~$\ref{th:BS-weak}$]
We have $\vrad(P_H(K)) \ls \vrad(P_H(K_{\rm out}))$, and by the Blaschke--Santal\'{o} inequality applied to $P_H(K_{\rm in})$ in $H$,
\[
\vrad(K^{\circ} \cap H) \ls \vrad((K_{\rm in})^{\circ} \cap H) \ls \vrad(P_H(K_{\rm in}))^{-1}.
\]
Therefore,
\[
\vrad(P_H(K)) \, \vrad(K^{\circ} \cap H)
   \ls \frac{\vrad(P_H(K_{\rm out}))}{\vrad(P_H(K_{\rm in}))}
   \ls \frac{n g(n)}{k},
\]
in view of Theorem~\ref{th:projections}.
\end{proof}

In the case where $K$ is isotropic, we can show that for a random $k$-dimensional subspace of $\mathbb{R}^n$, 
the volume radii of the corresponding sections or projections of $K_{\rm out}$ and $K_{\rm in}$ are of the same order, 
up to a logarithmic factor in the dimension.

\begin{proof}[Proof of Theorem~$\ref{th:random-subspace}$]
Let $K$ be an isotropic convex body in $\mathbb{R}^n$. 
Paouris and Pivovarov~\cite{Paouris-Pivovarov-2013} proved that for every symmetric convex body $C \subset \mathbb{R}^n$ 
and every $1 \ls k \ls n-1$,
\begin{equation}\label{eq:phi}
\Phi_k(C)
:= \vol_n(C)^{-1/n}
   \left(
     \int_{G_{n,k}} \vol_k(P_H(C))^{-n} \, d\nu_{n,k}(H)
   \right)^{-\frac{1}{kn}}
   \gr c_1 \sqrt{\frac{n}{k}},
\end{equation}
where $c_1 > 0$ is an absolute constant. 
Moreover, E.~Milman and Yehudayoff~\cite{EMilman-Yehudayoff} established the sharp lower bound 
$\Phi_k(B_2^n) \ls \Phi_k(C)$, together with a characterization of equality: 
ellipsoids are the only local minimizers with respect to the Hausdorff metric, for all $1 \ls k \ls n-1$.

From \eqref{eq:phi} and Markov’s inequality, it follows that
\[
{\rm vrad}(P_H(C)) \gr c_1 t^{-1} \sqrt{n} \, \vol_n(C)^{1/n}
\]
with probability greater than $1 - t^{-kn}$. Applying this result to $K_{\rm in}$ and taking into account the inequality of 
V.~Milman and Pajor, which ensures that $\vol_n(K_{\rm in})^{1/n} \gr \frac{1}{2}\vol_n(K)^{1/n} = \frac{1}{2}$, we obtain
\begin{equation}\label{eq:iso-proj-1}
{\rm vrad}(P_H(K_{\rm in})) \gr c_2 \sqrt{n}
\end{equation}
with probability greater than $1 - e^{-kn}$, where $c_2 > 0$ is an absolute constant.

Next, we use a well-known consequence of the Aleksandrov inequalities 
(see~\cite[Section~6.4]{Schneider-book}). 
For every convex body $C \subset \mathbb{R}^n$, the sequence
\begin{equation*}
Q_k(C)
= \left(
    \frac{1}{\omega_k}
    \int_{G_{n,k}} \vol_k(P_H(C)) \, d\nu_{n,k}(H)
  \right)^{1/k}
\end{equation*}
is decreasing in $k$. In particular, for any $1 \ls k \ls n-1$,
\[
Q_k(C) \ls Q_1(C),
\]
which can be written equivalently as
\begin{equation*}
\left(
  \frac{1}{\omega_k}
  \int_{G_{n,k}} \vol_k(P_H(C)) \, d\nu_{n,k}(H)
\right)^{1/k}
\ls M^{\ast}(C).
\end{equation*}
Applying this to $K_{\rm out}$ and using that 
$M^{\ast}(K_{\rm out}) \ls 2 M^{\ast}(K) \ls c \sqrt{n} (\ln n)^2$ 
by E.~Milman’s inequality, we obtain
\begin{equation}\label{eq:iso-proj-2}
{\rm vrad}(P_H(K_{\rm out})) \ls \gamma_n \sqrt{n}
\end{equation}
with probability greater than $1 - e^{-2k}$, where $\gamma_n \ls c_3 (\ln n)^2$. 
Combining \eqref{eq:iso-proj-1} and \eqref{eq:iso-proj-2}, we conclude that
\[
{\rm vrad}(P_H(K_{\rm out})) \ls \gamma_n \, {\rm vrad}(P_H(K_{\rm in}))
\]
with probability greater than $1 - e^{-k}$.

\medskip
\noindent
We now turn to the proof of \eqref{eq:iso-2}. 
Repeating the above reasoning for the polar body $K^{\circ}$, we obtain
\begin{equation}\label{eq:iso-proj-3}
{\rm vrad}(P_H((K^{\circ})_{\rm in})) 
  \gr c_2 \sqrt{n} \, \vol_n(K^{\circ})^{1/n}
\end{equation}
and
\begin{equation}\label{eq:iso-proj-4}
{\rm vrad}(P_H((K^{\circ})_{\rm out})) 
  \ls 2 M^{\ast}(K^{\circ})
  = 2 M(K)
  \ls \frac{\delta_n}{\sqrt{n}}
\end{equation}
for a random $H \in G_{n,k}$, where $\delta_n \ls c\,\ln n$. 
It follows that
\[
{\rm vrad}(P_H((K^{\circ})_{\rm out})) 
  \ls \delta_n \, {\rm vrad}(P_H((K^{\circ})_{\rm in}))
\]
with probability greater than $1 - e^{-k}$. 

Finally, repeating the proof of Theorem~\ref{th:sections}, 
and using the Blaschke--Santal\'{o} inequality for the symmetric convex bodies 
$K_{\rm out} \cap H$ and $K_{\rm in} \cap H$ 
(whose polars are $P_H((K^{\circ})_{\rm in})$ and $P_H((K^{\circ})_{\rm out})$, respectively), 
we obtain
\[
\vrad(K_{\rm out} \cap H) 
  \ls \delta_n \, \vrad(K_{\rm in} \cap H)
\]
for every $H \in G_{n,k}$ satisfying \eqref{eq:iso-proj-3} and \eqref{eq:iso-proj-4}.
\end{proof}

\section{Functional inequalities}\label{section:6}

We consider the class $\mathcal{L}^n$ of \emph{geometric log-concave integrable functions}:  
these are the centered log-concave functions $f:\mathbb{R}^n \to \mathbb{R}^+$ that satisfy
$f(0) = \|f\|_{\infty}$ and $0 < \int f < \infty$. We say that $f$ is centered if $\int xf(x)\,dx=0$. 
For any $f, g \in \mathcal{L}^n$ we define
\[
(f \star g)(x) = \sup\{ f(x_1) g(x_2) : x_1, x_2 \in \mathbb{R}^n,\ x = x_1 + x_2 \}.
\]
Roysdon~\cite{Roysdon-2020} obtained a functional version of Rudelson’s theorem on the sections of the difference body.  
He considered $\overline{f}(x) = f(-x)$ and defined $\Delta_0 f = f \star \overline{f}$, i.e.
\[
(\Delta_0 f)(x) = \sup\{ f(x_1) f(x_2) : x_1, x_2 \in \mathbb{R}^n,\ x = x_1 - x_2 \}.
\]
Note that if $f = \mathds{1}_K$ for some convex body $K \subset \mathbb{R}^n$, then 
$\Delta_0 f = \mathds{1}_{K - K}$.  
Thus, $\Delta_0 f$ may be viewed as a functional analogue of the difference body.  
With this definition, Roysdon’s inequality reads as follows.

\begin{theorem}[Roysdon]\label{th:Roysdon}
For any $f \in \mathcal{L}^n$, any $1 \ls k \ls n-1$, and any $H \in G_{n,k}$ we have
\begin{equation}\label{eq:Roysdon-1}
\left( \int_H \Delta_0 f(x) \, dx \right)^{1/k} 
  \ls C \max\!\left\{ \sqrt{k}, \frac{n}{k} \right\}
     \sup_{y \in \mathbb{R}^n}
     \left(
       \frac{1}{\|f|_{y+H}\|_{\infty}}
       \int_{y+H} f(x) \, dx
     \right)^{1/k},
\end{equation}
where $f|_{y+H}$ denotes the restriction of $f$ to $y+H$. Moreover,
\[
c \left( \int_H f(x) \, dx \right)^{1/k}
  \ls
  \left( \int_H \Delta_0 f(x) \, dx \right)^{1/k},
\]
where $c > 0$ and $C > 1$ are absolute constants.
\end{theorem}

\smallskip

For any $f \in \mathcal{L}^n$ we define the functions $\Delta_{\rm out} f$ and $\Delta_{\rm in} f$ by
\[
\Delta_{\rm out} f(x)
   = \sup\!\left\{ \sqrt{ f(x_1) f(-x_2) } :
      x = \tfrac{1}{2}(x_1 + x_2) \right\},
\qquad
\Delta_{\rm in} f(x)
   = \min\{ f(x), f(-x) \}.
\]
This definition is consistent with the notion of the \emph{$\alpha$-difference function} 
$\Delta_{\alpha} f$ of a function $f$, introduced by Colesanti in~\cite{Colesanti-2006} 
for any $f : \mathbb{R}^n \to [0, \infty]$ and any $\alpha \in [-\infty, 0]$. 
It is shown in \cite{Colesanti-2006} that if $f$ is $\alpha$-concave, then $\Delta_{\alpha} f$ is also $\alpha$-concave. 
Difference functions have been studied as functional analogues of the difference body in a number of works; 
see, e.g.,~\cite{Colesanti-2006, Klartag-2007, Bobkov-Colesanti-Fragala-2014}. 
One can easily check that the functions $\Delta_{\mathrm{out}} f$ and $\Delta_{\mathrm{in}} f$ 
considered here coincide with $\Delta_{0} f$ and $\Delta_{-\infty} f$, respectively. 
In particular, if $f \in \mathcal{L}^n$, then $\Delta_{\mathrm{out}} f$ is log-concave and 
$\Delta_{\mathrm{in}} f$ is quasi-concave.

Note that if $f = \mathds{1}_K$ for some convex body $K \subset \mathbb{R}^n$, then 
$\Delta_{\rm out} f = \mathds{1}_{\frac{1}{2}(K-K)}$ and 
$\Delta_{\rm in} f = \mathds{1}_{K \cap (-K)}$.  
Moreover, it is clear that $\Delta_{\rm in} f(x) \ls f(x)$ and 
$\Delta_{\rm in} f(x) \ls \Delta_{\rm out} f(x)$ for every $x \in \mathbb{R}^n$; 
for the latter inequality we choose $x_1 = x_2 = x$ and note that 
$\Delta_{\rm out} f(x) \gr \sqrt{f(x) f(-x)} \gr \min\{ f(x), f(-x) \}$.

\smallskip

We now establish the following functional version of Theorem~\ref{th:sections}.

\begin{theorem}\label{th:new-functional}
For any $f \in \mathcal{L}^n$, any $1 \ls k \ls n-1$, and any $H \in G_{n,k}$ we have
\begin{equation}\label{eq:new-functional-upper}
\left( \int_H \Delta_{\rm out} f(x) \, dx \right)^{1/k}
  \ls C (n/k)^2 g(n)
     \left( \int_H \Delta_{\rm in} f(x) \, dx \right)^{1/k}
\end{equation}
and
\begin{equation}\label{eq:new-functional-lower}
c \left( \int_H f(x) \, dx \right)^{1/k}
  \ls
  \left( \int_H \Delta_{\rm out} f(x) \, dx \right)^{1/k},
\end{equation}
where $c > 0$ and $C > 1$ are absolute constants.
\end{theorem}

\smallskip

For the proof we employ two families of convex bodies associated with each $f \in \mathcal{L}^n$.  
The first, introduced by K.~Ball~\cite{Ball-1988}, is given for every $p > 0$ by
\begin{equation*}
K_p(f)
  = \left\{ x \in \mathbb{R}^n :
      \int_0^{\infty} r^{p-1} f(r x) \, dr
      \gr \frac{f(0)}{p}
    \right\}.
\end{equation*}
From the definition it follows that the radial function of $K_p(f)$ satisfies
\begin{equation}\label{eq:radial-Kp}
\varrho_{K_p(f)}(x)
  = \left(
      \frac{1}{f(0)}
      \int_0^{\infty} p r^{p-1} f(r x) \, dr
    \right)^{1/p}
\quad \text{for } x \neq 0.
\end{equation}
Moreover, for every $0 < p < q$ one has
\begin{equation}\label{eq:inclusions-Kp}
\frac{\Gamma(p+1)^{1/p}}{\Gamma(q+1)^{1/q}} K_q(f)
  \subseteq K_p(f)
  \subseteq K_q(f).
\end{equation}
A proof of these inclusions is given in~\cite[Proposition~2.5.7]{BGVV-book}.  
We also consider the family $\{ R_p(f) \}_{p>1}$ defined by
\[
R_p(f) = \{ x \in \mathbb{R}^n : f(x) \gr e^{-(p-1)} f(0) \}.
\]
We will use several relations between these two families; their proofs are provided in 
Subsection~\ref{subsection:6.3}.  
Lemma~\ref{lem:6.3.2} shows that
\begin{equation}\label{eq:inclusions-K-R}
c_2 K_p(f) \subseteq R_p(f) \subseteq c_1 K_p(f)
\end{equation}
for all $p \gr 2$, where $c_1, c_2 > 0$ are absolute constants.  
Furthermore, Lemma~\ref{lem:6.3.3} states that
\begin{equation}\label{eq:inclusions-Delta-out-in}
K_p(\Delta_{\rm out} f)\approx K_p(f)_{{\rm out}}\quad\hbox{and}\quad 
K_p(\Delta_{\rm in} f)\approx K_p(f)_{{\rm in}}
\end{equation}
for all $p \gr 2$.

\smallskip

Since \eqref{eq:new-functional-upper} is homogeneous, we may assume that $f(0) = \|f\|_{\infty} = 1$.  
It is well known that for any $H \in G_{n,k}$,
\[
\int_H f(x) \, dx = \vol_k(K_k(f) \cap H).
\]
Indeed, integrating in polar coordinates gives
\begin{align*}
\vol_k(K_k(f) \cap H)
  &= \int_H \mathds{1}_{K_k(f)}(x) \, dx
   = \int_{S^{n-1} \cap H} \int_0^{\rho_{K_k(f)}(\xi)} r^{k-1} \, dr \, d\xi \\
  &= \frac{1}{k} \int_{S^{n-1} \cap H} \rho_{K_k(f)}(\xi)^k \, d\xi
   = \frac{1}{k} \int_{S^{n-1} \cap H}
       k \int_0^{\infty} f(r \xi) r^{k-1} \, dr \, d\xi \\
  &= \int_H f(z) \, dz.
\end{align*}

\smallskip

\begin{proof}[Proof of Theorem~$\ref{th:new-functional}$]
We begin with the convex body $K_k(\Delta_{\rm out} f)$.  
Note that $\Delta_{\rm out} f(0) = \|\Delta_{\rm out} f\|_{\infty} = \|f\|_{\infty} = 1$.  
We have
\[
\int_H \Delta_{\rm out} f(x) \, dx
  = \vol_k(K_k(\Delta_{\rm out} f) \cap H)
  \ls \vol_k(K_{n+1}(\Delta_{\rm out} f) \cap H),
\]
since $K_k(\Delta_{\rm out} f) \subseteq K_{n+1}(\Delta_{\rm out} f)$.  
By Lemma~\ref{lem:6.3.3},
\[
K_{n+1}(\Delta_{\rm out} f)
  \subseteq c\,K_{n+1}(f)_{{\rm out}}
\]
for some absolute constant $c > 0$.  
As $K_{n+1}(f)$ is centered, we deduce that
\begin{equation}\label{eq:new-functional-1}
\vol_k\big( K_{n+1}(f)_{{\rm out}} \cap H \big)
  \ls \left( \frac{n g(n)}{k} \right)^k
       \vol_k\big( K_{n+1}(f)_{{\rm in}} \cap H \big).
\end{equation}
Moreover,
\[
K_{n+1}(f)_{{\rm in}}
  \subseteq \frac{c n}{k} \, K_k(f)_{{\rm in}},
\]
and hence
\[
\vol_k(K_{n+1}(f)_{{\rm in}} \cap H)
  \ls \left( \frac{c n}{k} \right)^k
       \vol_k(K_k(f)_{{\rm in}} \cap H).
\]
Lemma~\ref{lem:6.3.3} implies that 
$K_k(\Delta_{\rm in} f) \approx K_k(f)_{{\rm in}}$, therefore
\[
\big( \vol_k(K_k(f)_{{\rm in}}\cap H) \big)^{1/k}
  \approx
  \big( \vol_k(K_k(\Delta_{\rm in} f) \cap H) \big)^{1/k}
  = \left( \int_H \Delta_{\rm in} f(x) \, dx \right)^{1/k}.
\]
Combining the above estimates gives
\[
\left( \int_H \Delta_{\rm out} f(x) \, dx \right)^{1/k}
  \ls (n/k)^2 g(n)
      \left( \int_H \Delta_{\rm in} f(x) \, dx \right)^{1/k},
\]
which completes the proof of~\eqref{eq:new-functional-upper}.

\smallskip
For~\eqref{eq:new-functional-lower}, observe first that
\[
K_k(f|_H) \subseteq c_1 (K_k(f) \cap H)
\]
for some absolute constant $c_1 > 0$.  
Indeed,
\begin{align*}
K_k(f|_H)
  &\subseteq c_2 R_k(f|_H)
   = c_2 \{ x \in H : f(x) \gr (f|_H)(0) e^{-(k-1)} \} \\
  &= c_2 \{ x \in H : f(x) \gr e^{-(k-1)} \}
   = c_2 \big( \{ x \in \mathbb{R}^n : f(x) \gr e^{-(k-1)} \} \cap H \big) \\
  &= c_2 (R_k(f) \cap H)
   \subseteq c_1 (K_k(f) \cap H),
\end{align*}
by~\eqref{eq:inclusions-K-R}.  
Then, combining Lemma~\ref{lem:6.3.3} with the Brunn--Minkowski inequality, we obtain
\begin{align}\label{eq:new-functional-2}
\vol_k(K_k(f|_H))
  &\ls c_1^k \vol_k(K_k(f) \cap H)
   \ls c_1^k 2^{-k} \vol_k(K_k(f)_{{\rm out}} \cap H) \\
\nonumber
  &\ls c_3^k \vol_k(K_k(\Delta_{\rm out} f) \cap H)
   = c_3^k \int_H \Delta_{\rm out} f(x) \, dx.
\end{align}
Thus,
\[
\int_H f(x) \, dx
  \ls c_3^k \int_H \Delta_{\rm out} f(x) \, dx,
\]
and taking $k$th roots completes the proof of~\eqref{eq:new-functional-lower}.
\end{proof}

Next, we establish a functional analogue of Theorem~\ref{th:projections}.

\begin{theorem}\label{th:functional-projections}
For any $f \in \mathcal{L}^n$, any $1 \ls k \ls n-1$, and any $H \in G_{n,k}$ we have
\begin{equation}\label{eq:functional-projections}
\left( \int_H P_H(\Delta_{\rm out} f)(x) \, dx \right)^{1/k}
  \ls C (n/k)^2 g(n)
     \left( \int_H P_H(\Delta_{\rm in} f)(x) \, dx \right)^{1/k},
\end{equation}
where $C > 0$ is an absolute constant.
\end{theorem}

\begin{proof}
Recall that if $g : \mathbb{R}^n \to [0,\infty)$ is a nonnegative measurable function and $H \in G_{n,k}$,  
the orthogonal projection of $g$ onto $H$ is the function $P_H g : H \to [0,\infty)$ defined by
\[
(P_H g)(z) = \sup\{ g(y + z) : y \in H^{\perp} \}.
\]
It is straightforward to verify that
\[
R_p(P_H g) = P_H(R_p(g))
\quad \text{for every } p > 1.
\]
Moreover, if $g = \mathds{1}_K$ for some compact set $K \subset \mathbb{R}^n$, then $P_H g = \mathds{1}_{P_H(K)}$.

\smallskip

We shall also use the fact that if $g(0) = \|g\|_{\infty} = 1$, then
\begin{align*}
\| P_H g \|_1
   &= \int_H P_H g(x) \, dx
    = \int_0^1 \vol_k(\{ x \in H : P_H g(x) \gr t \}) \, dt \\
   &= \int_1^{\infty} e^{-(p-1)} \vol_k(\{ x \in H : P_H g(x) \gr e^{-(p-1)} \}) \, dp \\
   &= \int_1^{\infty} e^{-(p-1)} \vol_k(R_p(P_H g)) \, dp
    = \int_1^{\infty} e^{-(p-1)} \vol_k(P_H(R_p(g))) \, dp.
\end{align*}

\smallskip

Let $f \in \mathcal{L}^n$ with $f(0) = \|f\|_{\infty} = 1$.  
By the right-hand side inclusion of Lemma~\ref{lem:6.3.2} we have
\begin{align}\label{eq:proj-1}
\int_H P_H(\Delta_{\rm out} f)(x) \, dx
  &= \int_1^{\infty} e^{-(p-1)} \vol_k(P_H(R_p(\Delta_{\rm out} f))) \, dp \\
\nonumber
  &\ls c_1^k \int_1^{\infty} e^{-(p-1)}
         \vol_k(P_H(K_p(\Delta_{\rm out} f))) \, dp.
\end{align}
From~\eqref{eq:inclusions-Kp} we know that
$K_p(\Delta_{\rm out} f) \subseteq K_{n+1}(\Delta_{\rm out} f)$ for all $1 \ls p \ls n+1$,  
and we also know that $K_p(\Delta_{\rm out} f) \subseteq \frac{c p}{n+1} K_{n+1}(\Delta_{\rm out} f)$ for all $p > n+1$.  
Therefore,
\begin{align}\label{eq:proj-2}
\int_H P_H(\Delta_{\rm out} f)(x) \, dx
  &\ls c_1^k \Bigg(
      \int_1^{n+1} e^{-(p-1)}
         \vol_k(P_H(K_{n+1}(\Delta_{\rm out} f))) \, dp \\
\nonumber
  &\hspace{1cm}
     + \int_{n+1}^{\infty}
         \left( \frac{c_2 p}{n+1} \right)^k e^{-(p-1)}
         \vol_k(P_H(K_{n+1}(\Delta_{\rm out} f))) \, dp
     \Bigg) \\
\nonumber
  &= c_1^k \, \vol_k(P_H(K_{n+1}(\Delta_{\rm out} f))) 
     \left(
       \int_1^{n+1} e^{-(p-1)} \, dp
       + \int_{n+1}^{\infty}
          \left( \frac{c_2 p}{n+1} \right)^k e^{-(p-1)} \, dp
     \right) \\
\nonumber
  &\ls c_1^k \, \vol_k(P_H(K_{n+1}(\Delta_{\rm out} f))) 
     \left(
       \int_1^{\infty} e^{-(p-1)} \, dp
       + \frac{e \, c_2^k}{(n+1)^k} \int_0^{\infty} p^k e^{-p} \, dp
     \right) \\
\nonumber
  &\ls c_1^k \, \vol_k(P_H(K_{n+1}(\Delta_{\rm out} f))) 
     \left( 1 + \frac{e \, c_2^k k!}{(n+1)^k} \right)
   \ls c_3^k \, \vol_k(P_H(K_{n+1}(\Delta_{\rm out} f))).
\end{align}

On the other hand,
\begin{align}\label{eq:proj-3}
\int_H P_H(\Delta_{\rm in} f)(x) \, dx
  &= \int_1^{\infty} e^{-(p-1)}
       \vol_k(P_H(R_p(\Delta_{\rm in} f))) \, dp \\
\nonumber
  &\gr c_4^k \int_2^{\infty} e^{-(p-1)}
       \vol_k(P_H(K_p(\Delta_{\rm in} f))) \, dp \\
\nonumber
  &\gr c_4^k \int_{\max\{k,2\}}^{n+1} e^{-(p-1)}
       \left( \frac{c_5 k}{n+1} \right)^k
       \vol_k(P_H(K_{n+1}(\Delta_{\rm in} f))) \, dp \\
\nonumber
  &\gr \left( \frac{c_6 k}{n+1} \right)^k
       \vol_k(P_H(K_{n+1}(\Delta_{\rm in} f))).
\end{align}

Since $K_{n+1}(f)$ is centered, Theorem~\ref{th:projections} gives
\[
\vol_k(P_H(K_{n+1}(\Delta_{\rm out} f)))^{1/k}
  \ls c_7 (n/k) g(n)
     \vol_k(P_H(K_{n+1}(\Delta_{\rm in} f)))^{1/k}.
\]
Combining this estimate with~\eqref{eq:proj-2} and~\eqref{eq:proj-3} yields
\[
\int_H P_H(\Delta_{\rm out} f)(x) \, dx
  \ls \left( \frac{c_8 n \sqrt{g(n)}}{k} \right)^{2k}
       \int_H P_H(\Delta_{\rm in} f)(x) \, dx,
\]
or equivalently,
\[
\left( \int_H P_H(\Delta_{\rm out} f)(x) \, dx \right)^{1/k}
  \ls c_8 (n/k)^2 g(n)
     \left( \int_H P_H(\Delta_{\rm in} f)(x) \, dx \right)^{1/k}.
\]
This completes the proof.
\end{proof}

\subsection{Auxiliary lemmas}\label{subsection:6.3}

In this subsection we present variants of several technical lemmas from~\cite{Klartag-VMilman-2005}
(see also~\cite{Roysdon-2020}).

\begin{lemma}\label{lem:6.3.1}
Let $g : [0, +\infty) \to [0, +\infty)$ be an increasing convex function with $g(0) = 0$.
For any $p > 1$ define
\[
M_p = \sup\{ e^{-g(t)} t^{p-1} : t > 0 \},
\]
and let $t_p > 0$ denote the unique point satisfying $M_p = e^{-g(t_p)} t_p^{p-1}$.
Then,
\begin{equation}\label{eq:6.3.1-1}
\frac{M_p t_p}{p}
   \ls \int_0^{\infty} t^{p-1} e^{-g(t)} \, dt
   \ls c \, \frac{M_p t_p}{\sqrt{p-1}},
\end{equation}
where $c > 0$ is an absolute constant. Moreover,
\begin{equation}\label{eq:6.3.1-2}
g(2t_p) \gr p - 1 \gr g(t_p).
\end{equation}
\end{lemma}

\begin{proof}
Since $g$ is increasing, we have
\[
\int_0^{\infty} t^{p-1} e^{-g(t)} \, dt
   \gr \int_0^{t_p} t^{p-1} e^{-g(t)} \, dt
   \gr e^{-g(t_p)} \int_0^{t_p} t^{p-1} \, dt
   = \frac{e^{-g(t_p)} t_p^p}{p}
   = \frac{M_p t_p}{p}.
\]

For the upper bound, set $\varphi(t) = g(t) - (p-1) \ln t$.  
Note that $\varphi$ is convex and has a unique critical point at $t_p$. Since $\varphi'_{-}(t_p) \ls 0 \ls \varphi'_{+}(t_p)$, it follows that
\[
g(t) \gr g(t_p) + \frac{p-1}{t_p} (t - t_p)
   \quad \text{for all } t > 0.
\]
Hence,
\begin{align*}
\int_0^{\infty} t^{p-1} e^{-g(t)} \, dt
   &\ls e^{(p-1) - g(t_p)} \int_0^{\infty} t^{p-1} e^{-\frac{(p-1)t}{t_p}} \, dt \\
   &= e^{(p-1) - g(t_p)} \left( \frac{t_p}{p-1} \right)^p \int_0^{\infty} t^{p-1} e^{-t} \, dt \\
   &= e^{-g(t_p)} t_p^{p-1}
      \frac{e^{p-1} (p-1)!}{(p-1)^{p-1}}
      \frac{t_p}{p-1}
      \approx M_p \frac{t_p}{\sqrt{p-1}}.
\end{align*}

Finally, if $0 < t < t_p$ then $g'_+(t) \ls (p-1)/t_p$, which implies
\[
g(t_p) \ls g(0) + \int_0^{t_p} \frac{p-1}{t_p} \, dt = p - 1,
\]
while
\[
g(2t_p) \gr g(t_p) + \frac{p-1}{t_p}(2t_p - t_p) \gr p - 1.
\]
This completes the proof.
\end{proof}

\begin{lemma}\label{lem:6.3.2}
Let $f \in \mathcal{L}^n$ with $f(0)=\|f\|_{\infty}=1$. For any $p \gr 2$ we have
\[
c_1 K_p(f) \subseteq R_p(f) \subseteq c_2 K_p(f),
\]
where $c_1, c_2 > 0$ are absolute constants. The right-hand side inclusion holds for all $p>1$.
\end{lemma}

\begin{proof}
Fix $\xi \in S^{n-1}$ and consider the convex function $g(t) = -\ln f(t\xi)$.
Let $M_p = \sup\{ e^{-g(t)} t^{p-1} : t > 0 \}$, and suppose that
$e^{-g(t_p)} t_p^{p-1} = f(t_p) t_p^{p-1} = M_p$.
Then, by Lemma~\ref{lem:6.3.1},
\[
\frac{M_p t_p}{p}
   \ls \int_0^{\infty} f(t) t^{p-1} \, dt
   \ls c \, \frac{M_p t_p}{\sqrt{p-1}}.
\]
By the definition of $K_p(f)$, this implies
\[
(M_p t_p)^{1/p}
   \ls \varrho_{K_p(f)}(\xi)
   \ls (M_p t_p)^{1/p} \left( \frac{cp}{\sqrt{p-1}} \right)^{1/p}
   \ls c_1 (M_p t_p)^{1/p}
\]
for all $p\gr 2$, where $c_1>0$ is an absolute constant. The left-hand side inequality holds for all $p>1$.

Since $M_p t_p = f(t_p\xi) t_p^p \ls t_p^p$, and from Lemma~\ref{lem:6.3.1} we have
$g(t_p) \ls p - 1$, it follows that
\[
M_p t_p = f(t_p\xi) t_p^{p-1} t_p \gr e^{-(p-1)} t_p^p\gr e^{-p}t_p^p.
\]
Thus,
\[
c_2 t_p \ls \varrho_{K_p(f)}(\xi) \ls c_1 t_p.
\]
Moreover, since $g(t_p) \ls p - 1 \ls g(2t_p)$, we also have
$f(2t_p\xi) \ls e^{-(p-1)} \ls f(t_p\xi)$, implying
\[
t_p \ls \varrho_{R_p(f)}(\xi) \ls 2t_p.
\]
Combining these estimates yields the desired inclusion.
\end{proof}

\begin{lemma}\label{lem:6.3.3}
Let $f \in \mathcal{L}^n$ with $f(0) = \|f\|_{\infty} = 1$.  
For any $p \gr 2$ we have
\[
K_p(\Delta_{\rm out} f) \approx K_p(f)_{{\rm out}}
\qquad
K_p(\Delta_{\rm in} f) \approx K_p(f)_{{\rm in}}.
\]
\end{lemma}

\begin{proof}
Let $x \in K_p(\Delta_{\rm out} f)$.  
By Lemma~\ref{lem:6.3.2}, $c_1x \in R_p(\Delta_{\rm out} f)$.
Hence there exist $x_1, x_2 \in \mathbb{R}^n$ such that
$c_1x = \tfrac{1}{2}(x_1 + x_2)$ and
$f(x_1) f(-x_2) \gr e^{-2(p-1)}$.
Since $\|f \|_{\infty} = 1$, we obtain
$f(x_1) \gr e^{-2(p-1)}$ and $f(-x_2) \gr e^{-2(p-1)}$,
that is, $x_1 \in R_{2p-1}(f)$ and $x_2 \in R_{2p-1}(\overline{f})$.
Applying Lemma~\ref{lem:6.3.2} again gives
\[
c_1 x \in \tfrac{1}{2}\big(R_{2p-1}(f) + R_{2p-1}(\overline{f})\big)
   \subseteq c_2 \big( K_{2p-1}(f) - K_{2p-1}(f) \big)\approx K_{2p-1}(f)_{{\rm put}}\approx K_p(f)_{{\rm out}},
\]
where we also used $K_p(\overline{f}) = -K_p(f)$ and \eqref{eq:inclusions-Kp}.
Thus,
\[
K_p(\Delta_{\rm out} f)
   \subseteq c_3 K_p(f)_{{\rm out}}.
\]
Conversely, if $x_1, x_2 \in K_p(f)$, then $c_1 x_1 \in R_p(f)$ and
$-c_1 x_2 \in R_p(\overline{f})$. Hence,
\[
(\Delta_{\rm out} f)(c_1(x_1 - x_2))
   \gr \sqrt{ f(c_1 x_1) f(-c_1 x_2) }
   \gr e^{-(p-1)},
\]
which implies
\[
x_1 - x_2 \in c_1^{-1} R_p(\Delta_{\rm out} f)
   \subseteq (c_2 / c_1) K_p(\Delta_{\rm out} f).
\]
Therefore,
\[
K_p(f)_{{\rm out}}\subseteq K_p(f) - K_p(f) \subseteq c_4 K_p(\Delta_{\rm out} f).
\]

For the second assertion, observe that
\begin{align*}
R_p(\Delta_{\rm in} f)
   &= \{ x : \min\{ f(x), \overline{f}(x) \} \gr e^{-(p-1)} \} \\
   &= \{ x : f(x) \gr e^{-(p-1)} \}
      \cap \{ x : \overline{f}(x) \gr e^{-(p-1)} \}
      = R_p(f) \cap R_p(\overline{f}),
\end{align*}
and thus, by Lemma~\ref{lem:6.3.2},
\[
K_p(\Delta_{\rm in} f)
   \approx K_p(f) \cap K_p(\overline{f})
   = K_p(f) \cap (-K_p(f))=K_p(f)_{{\rm in}}.
\]
\end{proof}

\bigskip

\noindent {\bf Acknowledgement.} The author acknowledges support by a PhD scholarship
from the National Technical University of Athens.

\bigskip

\footnotesize
\bibliographystyle{amsplain}

\bigskip

\thanks{\noindent {\bf Keywords:} convex bodies, sections and projections, isotropic position, $MM^{\ast}$-estimate,
log-concave functions.

\smallskip

\thanks{\noindent {\bf 2020 MSC:} Primary 52A23; Secondary 46B06, 52A40, 26B25.}

\bigskip

\noindent \textsc{Natalia \ Tziotziou}: School of Applied Mathematical and Physical Sciences, National Technical University of Athens, Department of Mathematics, Zografou Campus, GR-157 80, Athens, Greece.

\smallskip

\noindent \textit{E-mail:} \texttt{nataliatz99@gmail.com}

\end{document}